\documentclass[12pt]{amsart}
\usepackage{amscd,amssymb}
\usepackage[dvips]{graphics}

\newtheorem{thm}{Theorem} [section]
\newtheorem{cor}[thm]{Corollary}
\newtheorem{lemma}[thm]{Lemma}
\newtheorem*{conj*}{Conjecture}
\newtheorem*{thm*}{Theorem}
\newtheorem*{thmA*}{Theorem A}
\newtheorem*{thmB*}{Theorem B}
\newtheorem*{lemma*}{Lemma}
\newtheorem{defn}[thm]{Definition}

\newcommand{\R}{\mathbb R}

\newcommand{\THS}{\Sigma} 

\newcommand{\oto}{\mathsf{O}(2,1)}
\newcommand{\soto}{{\mathsf{SO}(2,1)^0}}

\newcommand{\iso}{{\mathsf{G}}}

\newcommand{\Aff}{{\mathsf{Aff}}} 
\newcommand{\GLthR}{{\mathsf{GL}(3,\R)}}
\newcommand{\Ht}{{\mathbf H}^2} 
\newcommand{\HS}{{\mathfrak H}} 

\newcommand{\Det}{\mathsf{Det}}

\newcommand{\s}{\mathfrak{s}} 

\newcommand{\EE}{{\mathbf E}}  
\newcommand{\E}{{\EE^3_1}}  
\newcommand{\V}{{\R^3_1}}  

\newcommand{\Hyp}{{\mathsf{Hyp}}} 
\newcommand{\Par}{{\mathsf{Par}}} 
\newcommand{\na}[1]{\tilde{\alpha}_{#1}}
\newcommand{\Fix}{{\mathsf{Fix}}} 
\newcommand{\G}{{\Gamma}} 
\newcommand{\Id}{\mathbb I}
\newcommand{\LL}{\mathsf L} 
\newcommand{\Z}{\mathbb Z}

\renewcommand{\P}{\mathcal P}

\newcommand{\CP}{{\mathcal C}}
\newcommand{\pp}{\P}

\renewcommand{\H}{\mathsf H}

\newcommand{\vx}{{\mathsf x}}
\newcommand{\vy}{{\mathsf y}}

\newcommand{\vu}{{\mathsf u}}
\newcommand{\vv}{{\mathsf v}}
\newcommand{\vw}{{\mathsf w}}

\newcommand{\vo}[1]{{\mathsf x}^0_{#1}}
\newcommand{\vp}[1]{{\mathsf x}^+_{#1}}
\newcommand{\vm}[1]{{\mathsf x}^-_{#1}}
\newcommand{\vpm}[1]{{\mathsf x}^{\pm}_{#1}}

\newcommand{\yp}[1]{{#1}^+ }
\newcommand{\ym}[1]{{#1}^- }
\newcommand{\ypm}[1]{{#1}^{\pm}}

\newcommand{\ldot}[2]{{#1} \cdot {#2}}

\renewcommand{\gg}{\gamma}

\renewcommand{\int}{\mathsf{int}} 

\newcommand{\ZZ}{\mathsf{Z}^1(\G_0,\V)}
\newcommand{\HH}{\H^1(\G_0,\V)}

\newcommand{\trho}{\tilde{\rho}}
\newcommand{\tr}{\mathsf{tr}}

\newcommand{\e}{{\mathsf e}}
\newcommand{\Lo}{{L_{0}}}
\newcommand{\Li}{{L_{\infty}}}

\newcommand{\GLtwR}{{\mathsf{GL}(2,\R)}}
\newcommand{\St}{{\mathsf S}_2}
\newcommand{\Spfr}{{\mathsf{Sp}(4,\R)}}
\newcommand{\Spfz}{{\mathsf{Sp}(4,\Z)}}

\newcommand{\SLtz}{\mathsf{SL}(2,\Z)}
\newcommand{\SLtr}{\mathsf{SL}(2,\R)}
\newcommand{\Lag}{\mathfrak{L}}
\newcommand{\Lf}{\Lag_\infty}

\newcommand{\Hom}{\mathsf{Hom}}

\newcommand{\Aut}{{\mathsf{Aut}}} %
\newcommand{\SAut}{{\mathsf{SAut}}}
\newcommand{\U}{{\mathsf{U}}}

\newcommand\muchbigger{>\hspace{-3pt}>}
\newcommand{\Listar}{L^*_\infty}

\begin{document}

\title[The three-holed sphere]
{Affine deformations of a three-holed sphere}

\author[Charette]{Virginie Charette}
    \address{D\'epartement de math\'ematiques\\ Universit\'e de Sherbrooke\\
    Sherbrooke, Quebec, Canada}
    \email{v.charette@usherbrooke.ca}
\author[Drumm]{Todd A. Drumm}
    \address{Department of Mathematics\\ Howard University\\
    Washington, DC }
    \email{tdrumm@howard.edu}
\author[Goldman]{William M. Goldman}
    \address{Department of Mathematics\\ University of Maryland\\
    College Park, MD 20742 USA}
    \email{wmg@math.umd.edu}

\date{\today}

\subjclass{57M05 (Low-dimensional topology), 20H10 (Fuchsian groups and their
generalizations), 30F60 (Teichm\"uller theory)}
\keywords{
hyperbolic surface, affine manifold, discrete group, fundamental
polygon, fundamental polyhedron, proper action, Lorentz metric,
Fricke space}

\thanks{Charette gratefully acknowledges partial support from the Natural Sciences and Engineering Research Council of Canada and from the Fonds qu\'eb\'ecois de la recherche sur la nature et les technologies. Goldman gratefully
acknowledges partial support from National Science Foundation
grants DMS070781 and the Oswald Veblen Fund at the Institute
for Advanced Study.}

\begin{abstract}
Associated to every complete affine $3$-manifold $M$ with
nonsolvable fundamental group is a  noncompact hyperbolic surface
$\Sigma$. We classify such complete affine structures when
$\Sigma$ is homeomorphic to a three-holed sphere. In particular,
for every such complete hyperbolic surface $\Sigma$, the
deformation space identifies with two opposite octants in $\R^3$.
Furthermore every $M$ admits a fundamental polyhedron bounded by
crooked planes. Therefore $M$ is homeomorphic to an open solid
handlebody of genus two. As an explicit application of this
theory, we construct proper affine deformations of an arithmetic
Fuchsian group inside $\Spfz$.
\end{abstract}

\maketitle

\tableofcontents
\section*{Introduction}

A {\em complete affine manifold\/} is a quotient
\begin{equation*}
M = \EE/\Gamma
\end{equation*}
where $\EE$ is an affine space and $\Gamma\subset \Aff(\EE)$ is a
discrete group of affine transformations of $\EE$ acting properly
and freely on $\EE$. When $\dim \EE =3$,
Fried-Goldman~\cite{FG} and Mess~\cite{Me} imply that either:
\begin{itemize}
\item $\Gamma$ is solvable, or
\item $\Gamma$ is virtually free.
\end{itemize}
When $\Gamma$ is solvable, $M$ admits a finite covering
homeomorphic to the total space of a fibration composed of points,
circles, annuli and tori. The classification of such structures in
this case is straightforward~\cite{FG}. When $\Gamma$ is virtually
free, the classification is considerably more interesting. In the
early 1980s Margulis discovered~\cite{Margulis1,Margulis2} the
existence of such structures, answering a question posed by
Milnor~\cite{Mi}.

\begin{conj*}
Suppose $M^3$ is a $3$-dimensional complete affine manifold
with free fundamental group. Then $M$ is homeomorphic to an open solid handlebody.
\end{conj*}

\noindent
The purpose of this paper is to prove this conjecture in the first
nontrivial case.

By Fried-Goldman~\cite{FG},
the linear holonomy homomorphism
\begin{equation*}
\Aff(\EE^3) \xrightarrow{\LL} \GLthR
\end{equation*}
 embeds $\Gamma$ as a discrete
subgroup of a subgroup of $\GLthR$ conjugate to the orthogonal
group $\oto$.
Thus
$M$ admits a {\em complete flat Lorentz metric \/}
and is a {\em (geodesically) complete flat Lorentz $3$-manifold.\/}
Thus we henceforth restrict our attention to the case $\EE$ is
a $3$-dimensional {\em Lorentzian affine space\/} $\E$.
A Lorentzian affine space is a simply connected
geodesically complete flat Lorentz $3$-manifold,
and is unique up to isometry.

Furthermore $\LL(\Gamma)$ is a Fuchsian group acting
properly and freely on the hyperbolic plane $\Ht$.
We model $\Ht$ on a component
of the two-sheeted hyperboloid
\begin{equation*}
\{ \vv\in\V \mid \ldot{\vv}{\vv} \,=\, -1 \},
\end{equation*}
or equivalently its projectivization in $\mathsf{P}(\V)$.
(Compare \cite{G}.)
The quotient
\begin{equation*}
\Sigma := \Ht/\LL(\Gamma)
\end{equation*}
is a complete hyperbolic surface homotopy-equivalent to $M$,
naturally associated to the Lorentz manifold $M$.

We prove the above conjecture in the case that the surface $\Sigma$ is
homeomorphic to a three-holed sphere.

Margulis~\cite{Margulis1,Margulis2} discovered proper actions by
bounding (from below) the Euclidean distance that elements of $\Gamma$ displace points.
Our more geometric approach constructs fundamental polyhedra
for affine deformations
in the spirit of Poincar\'e's theorem on fundamental polyhedra for
hyperbolic manifolds.

This approach began with Drumm~\cite{Drumm1}, who constructed
fundamental polyhedra from {\em crooked planes\/} to show that certain
affine deformations $\Gamma$ acts properly on all of $\E$.
A {\em crooked plane\/} is a polyhedron in $\E$ with four infinite faces,
adapted to the invariant Lorentzian geometry of $\E$.
Specifically, representing the
hyperbolic surface $\Sigma$ as an identification space of a
fundamental polygon for the generalized
Schottky group $\LL(\Gamma)\subset \oto$, we construct a
fundamental polyhedron for certain affine deformations $\Gamma$
bounded by crooked planes \cite{Drumm1}. We call such a
fundamental polyhedron a {\em crooked fundamental polyhedron.\/}

\begin{conj*}
Suppose $\dim(\E) = 3$ and $\Gamma\subset\Aff(\E)$ is
a discrete group acting properly on $\E$. Suppose that $\Gamma$ is
not solvable. Then some finite-index subgroup of $\Gamma$ admits a crooked fundamental domain.
\end{conj*}

\noindent
We prove this conjecture when $\Sigma$ is homeomorphic to a three-holed sphere.

Let $\Gamma_0\subset \oto$ be a Fuchsian group. Denote the corresponding embedding
\begin{equation*}
\rho_0: \Gamma_0 \hookrightarrow \oto  \subset \GLthR.
\end{equation*}
An {\em affine deformation\/} of $\Gamma_0$ is a representation
\begin{equation*}
\Gamma_0 \xrightarrow{\rho} \Aff(\E)
\end{equation*}
satisfying $L\circ\rho = \rho_0$.
We refer to the image $\Gamma$ of $\rho$ as an
{\em affine deformation\/} as well.

An  affine deformation is {\em proper\/} if the
affine action of $\Gamma_0$ on $\E$ defined by $\rho$ is
a proper action. Clearly an affine deformation $\Gamma$
which admits a crooked fundamental polyhedron is proper.

\begin{thm*}[Drumm]
Every free discrete Fuchsian group $\Gamma_0\subset \oto$ admits a
proper affine deformation.
\end{thm*}
\noindent Actions of free groups by Lorentz isometries are the
only cases to consider.  Fried-Goldman \cite{FG} reduces the
problem to when $\Gamma_0$ is a Fuchsian group, and Mess~\cite{Me}
implies $\Gamma_0$ cannot be cocompact. Thus, after passing to a
finite-index subgroup, we may assume that $\Gamma_0$ is free.

The linear representation $\rho_0$ is itself an affine deformation,
by composing it with the embedding
\begin{equation*}
\GLthR \hookrightarrow \Aff(\E).
\end{equation*}
Slightly abusing notation, denote this composition by $\rho_0$ as
well.  Two affine deformations are {\em translationally
equivalent\/} if they are conjugate by a translation in $\E$. An
affine deformation is {\em trivial\/} (or {\em radiant\/}) if and
only if it is translationally conjugate to the affine deformation
$\rho_0$ constructed above. Equivalently, an affine deformation is
trivial if it fixes a point in the affine space $\E$.

Let  $\V$ denote the vector space underlying the affine space $\E$, considered as a
$\Gamma_0$-module via the linear representation $\rho_0$.
The space of translational equivalence classes of affine deformations
of $\rho_0$ identifies with the cohomology group $\HH$.
For each $g\in\Gamma_0$,
define
the {\em translational part\/}
$u(g)$ of $\rho(g)$,
as the unique translation taking the origin
to its image under $\rho(g)$.
That is, $u(g) = \rho(g) (0)$, and
\begin{equation*}
x \stackrel{\rho(g)}\longmapsto \rho_0(g)(x) + u(g).
\end{equation*}
The map $\Gamma_0\xrightarrow{u}\V$ is a cocycle in $\ZZ$,
and conjugating $\rho$ by a translation changes $u$ by
a coboundary.

The classification of complete affine structures in dimension $3$
therefore reduces to determining, for a given free Fuchsian group
$\Gamma_0$, the subset of $\HH$ corresponding to translational
equivalence classes of {\em proper\/}  affine deformations.

Margulis~\cite{Margulis1,Margulis2} introduced an invariant of the
affine deformation $\Gamma$, defined for elements $\gamma\in\Gamma$
whose linear part $\LL(\gamma)$ is hyperbolic.
Namely, $\gamma$ preserves a unique affine line $C_\gamma$ upon which
it acts by translation.  Furthermore $C_\gamma$ inherits a
{\em  canonical orientation.\/} As $C_\gamma$ is spacelike, the Lorentz
metric and the canonical orientation determines a unique
orientation-preserving isometry
\begin{equation*}
\R\xrightarrow{j_\gamma}  C_\gamma.
\end{equation*}
\noindent
The {\em Margulis invariant\/} $\alpha(\gamma)\in\R$ is
the displacement of the translation $\gamma|_{C_\gamma}$ as measured by $j_\gamma$:
\begin{equation*}
j_\gamma(t) \xrightarrow{\gamma}  j_\gamma(t + \alpha(\gamma) \big)
\end{equation*}
for $t\in\R$.

Margulis's invariant $\alpha$ is a class function on $\Gamma_0$ which
completely determines the translational equivalence class of the
affine deformation \cite{DrummGoldman2, CharetteDrumm2}.
Charette and Drumm~\cite{CharetteDrumm1} extended Margulis's invariant to
parabolic transformations.  However, only its {\em sign\/} is well
defined for parabolic transformations.
To obtain a precise numerical value one requires a {\em decoration\/} of $\Gamma_0$,
that is, a choice of horocycle at each cusp of $\Sigma$.

If $\Gamma$ is an affine deformation of $\Gamma_0$ with translational
part $u\in\ZZ$, then we indicate the dependence of $\alpha$ on the
cohomology class $[u]\in\HH$ by writing $\alpha = \alpha_{[u]}$.

Let $\Gamma_0$  be a Fuchsian group whose corresponding hyperbolic surface
$\Sigma$ is homeomorphic to a three-holed sphere.
Denote the generators of $\Gamma_0$ corresponding to the three ends
of $\partial\Sigma$ by $g_1, g_2, g_3$.
Choose a decoration  so that the generalized Margulis invariant
defines an isomorphism
\begin{align*}
\HH &\longrightarrow \R^3 \\
[u] &\longmapsto
\bmatrix \mu_1([u]) \\ \mu_2([u]) \\ \mu_3([u]) \endbmatrix :=
\bmatrix
\alpha_{[u]}(g_1) \\ \alpha_{[u]}(g_2) \\ \alpha_{[u]}(g_3)
\endbmatrix.
\end{align*}

\begin{thmA*}\label{thm:main}
Let $\Gamma_0,\Sigma_0,\mu_1,\mu_2,\mu_3$ be as above.
Then $[u]\in\HH$
corresponds to a proper affine
deformation if and only if
\begin{equation*}
\mu_1([u]),\;  \mu_2([u]),\;  \mu_3([u])
\end{equation*}
all have the same sign.
Furthermore in this case $\Gamma$ admits a crooked
fundamental domain and $M$ is homeomorphic to an open solid
handlebody of genus two.
\end{thmA*}

For purely hyperbolic $\Gamma_0$,
Theorem~A
was proved by Cathy Jones in her doctoral thesis~\cite{Jones},
using a different method.

In the case that $\Sigma$ is a three-holed sphere, Theorem~A gives
a complete description of the deformation space and the
topological type. As three-holed spheres are the building blocks
of all compact hyperbolic surfaces, the present paper plays a
fundamental role in our investigation of affine deformations of
hyperbolic surfaces of arbitrary topological type. 
We conjecture that when 
$\Sigma$ is homeomorphic to a
two-holed projective plane or one-holed Klein bottle, 
the deformation space will again be defined by finitely
many inequalities.
However, in all other cases,
the deformation space will be
defined by infinitely many inequalities. 
For example, when $\Sigma$ is homeomorphic to a
one-holed torus, the deformation space is a convex domain with
fractal boundary~\cite{GMM}.

Margulis's opposite sign lemma~\cite{Margulis1,Margulis2} (see Abels~\cite{Abels}
for a beautiful exposition) states that uniform positivity (or negativity) of
$\alpha(\gamma)$ is necessary for properness of an affine deformation.
In \cite{GM,G0} uniform positivity was conjectured
to be equivalent to properness. Theorem~A implies this conjecture when $\Sigma$
is a three-holed sphere with geodesic boundary.
In that case only the three $\gamma$ corresponding to $\partial \Sigma$
need to be checked.
However, when $\Sigma$ has at least one cusp, Theorem~A
provides counterexamples to the original conjecture.
If the generalized Margulis invariant of that cusp is zero,
and those of the other ends are positive, then $\alpha(\gamma) > 0$
for all hyperbolic elements $\gamma\in\Gamma$.
Other counterexamples are given in \cite{GMM}.

We apply Theorem~A to construct a proper affine deformation of an
arithmetic group in $\SLtz$ inside $\Spfz$.  Here $\Aff(\E)$ is
represented as the subgroup of $\Spfr$ stabilizing a Lagrangian plane
$\Li$ 
in a symplectic vector space $\R^4$ defined over $\Z$.  
Its unipotent radical $\U$ is the subgroup of
$\Spfr$ which preserves $\Li$, 
acts identically on $\Li$,
and acts identically on the quotient $\R^4/\Li$.
The parabolic subgroup
$\Aff(\E)$ is the normalizer of $U$ in $\Spfr$.
Furthermore
$\Aff(\E)$ acts conformally on a left-invariant flat Lorentz
metric on $U$.
This model of
Minkowski space embeds in the conformal compactification of $\E$, the
{\em Einstein universe\/} (see \cite{BCDGM}) upon which $\Spfr$ acts
transitively.

\begin{thmB*}
Choose three positive integers $\mu_1,\mu_2,\mu_3$.
Let $\Gamma$ be the subgroup of $\Spfz$ generated by
\begin{equation*}
\bmatrix
-1 & -2 & \mu_1 + \mu_2 -\mu_3  & 0 \\
 0 & -1 & 2\mu_1 & -\mu_1 \\
 0 & 0 & -1   &   0 \\
 0 & 0 &  2   &   -1
\endbmatrix ,\;
\bmatrix
-1 &  0 & -\mu_2 & -2\mu_2    \\
 2 & -1 &     0 & 0 \\
 0 & 0  &      -1 & -2   \\
 0 & 0  &      0 & -1
\endbmatrix
\end{equation*}
Let $\U \,:=\, \exp(\Phi) \,\subset\, \Spfr$ 
be the connected unipotent subgroup consisting of matrices
\begin{equation*}
\bmatrix
1 & 0 & x & y \\
0 & 1 & y & z \\
0 & 0 & 1 & 0 \\
0 & 0 & 0 & 1
\endbmatrix
\end{equation*}
where $x,y,z\in\R$. Then:
\begin{itemize}
\item $\Gamma$ normalizes $\U$;
\item The resulting action of $\Gamma$ on $\U$
is proper and free;
\item $\Gamma$ acts isometrically with respect to a left-invariant
flat Lorentz metric on $\U$;
\item The quotient orbifold $\U/\Gamma$ is homeomorphic
to an open solid handlebody of genus two.
\end{itemize}
\end{thmB*}

Our result complements 
Goldman-La\-bourie-Mar\-gulis~\cite{GLM}
when the hyperbolic surface $\Sigma$ is convex cocompact.
In that case the space of proper affine deformations identifies with
an open convex cone in $\HH$ defined by the
nonvanishing of an extension of  Margulis's invariant to
geodesic currents on $\Sigma$.

This cone is the interior of the intersection of half-spaces defined
by the functionals
\begin{align*}
\HH & \longrightarrow \R \\
[u] &\longmapsto \alpha_{[u]} (g)
\end{align*}
for $g\in\Gamma_0$.
In general we expect this cone to be the {\em union\/} of open regions corresponding
to combinatorial configurations realized by crooked planes,
thereby giving a crooked fundamental domain for each proper affine deformation.
Jones~\cite{Jones} used standard {\em Schottky fundamental domains\/} to fill the
open cone with such regions.
Here we decompose $\Sigma$ into two ideal triangles,
obtaining a single combinatorial configuration which applies to all proper affine deformations.

We are grateful to Ian Agol, Francis Bonahon, Dick Canary, David
Gabai, Ryan Hoban, Cathy Jones, Fran\c cois Labourie, Misha Kapovich,
Grisha Margulis, Yair Minsky and Kevin Scannell for helpful
discussions.  We also wish to thank the Institute for Advanced Study
for their hospitality.

\section{Lorentzian geometry}\label{sec:alpha}

This section summarizes needed technical background on the geometry of Minkowski
(2+1)-spacetime, its isometries and Margulis's invariant
of hyperbolic and parabolic isometries.
For details, variants and proofs, see
\cite{Abels,CharetteDrumm1,CharetteDrumm2,
CDGM,Drumm3,DGyellow,DrummGoldman2,G0}.

Let $\E$ denote {\em Minkowski (2+1)-spacetime}, that is, a simply
connected complete three-dimensional flat Lorentzian
manifold. Alternatively $\E$ is an affine space whose underlying
vector space $\V$ of translations is a {\em Lorentzian inner vector space,\/}
a vector space with an inner product
\begin{align*}
\V \times \V &\longrightarrow \R \\
(\vv,\vw) &\longmapsto \vv\cdot\vw
\end{align*}
of signature $(2,1)$.

A vector $\vx\in\V$ is:
\begin{itemize}
\item {\em null} if $\ldot{\vx}{\vx}=0$;
\item {\em timelike} if $\ldot{\vx}{\vx}<0$;
\item {\em spacelike} if $\ldot{\vx}{\vx}<0$.
\end{itemize}
A spacelike vector $\vx$ is {\em unit spacelike} if $\ldot{\vx}{\vx}=1$.  
A null vector is {\em future-pointing} if its third coordinate is
positive -- this corresponds to choosing a connected component of
the set of timelike vectors, or a {\em time-orientation}.

Define the {\em Lorentzian cross-product} as follows.
Choose an orientation on $\V$.
Let
\begin{align*}
\V \times \V \times \V \xrightarrow{\Det} \R
\end{align*}
denote the alternating trilinear form compatible
with the Lorentzian inner product and the orientation:
if $(\vv_1,\vv_2,\vv_3)$ is a positively-oriented
basis, with
\begin{equation*}
\ldot{ \vv_i}{\vv_j} \,=\, 0 {\text{~if~}} i \neq j, \;
\ldot{\vv_1}{\vv_1} \,=\, \ldot{\vv_2}{\vv_2} \,=\, -\ldot{\vv_3}{\vv_3} \,=\, 1
\end{equation*}
then
\begin{equation*}
\Det(\vv_1,\vv_2,\vv_3) = 1.
\end{equation*}
The Lorentzian cross-product is the unique bilinear map
\begin{equation*}
\V\times\V\xrightarrow{\boxtimes} \V
\end{equation*}
satisfying
\begin{equation*}
\ldot{\vu}({\vv\boxtimes\vw}) \;=\;
\mbox{Det}(\left[\vu~ \vv~\vw\right]).
\end{equation*}
The following facts are well known
(see for example Ratcliffe~\cite{Ratcliffe}):
\begin{lemma}\label{lem:crossprod}
Let $\vu,\vv,\vx,\vy\in\V$.  Then:
\begin{align*}
\ldot{\vu}{(\vx\boxtimes\vy)} & = \ldot{\vx}{(\vy\boxtimes\vu)} \\
\ldot{(\vu\boxtimes\vv)}{(\vx\boxtimes\vy)} &=
(\ldot{\vu}{\vy})(\ldot{\vv}{\vx})-(\ldot{\vu}{\vx})(\ldot{\vv}{\vy}).
\end{align*}
\end{lemma}

For a spacelike vector $\vv$, define its {\em Lorentz-orthogonal plane} to be:
\begin{equation*}
\vv^\perp = \{ \vx \, \mid  \ldot{\vx}{\vv}=0\} .
\end{equation*}
It is an {\em indefinite plane}, since the Lorentzian inner
product restricts to an inner product of signature $(1,1)$.  In
particular, $\vv^\perp$ contains two null lines.  The two
future-pointing linear independent vectors of Euclidean length $1$
in this set are denoted $\vv^-$ and $\vv^+$ and are chosen so that
$( \vv^-,\vv^+, \vv )$ is a positively oriented basis for $\V$.

A basis $(a,b,c)$ of $\V$ is positively oriented if and only if
\begin{equation*}
\ldot{(a \boxtimes b)}{c} > 0 .
\end{equation*}

\begin{lemma}\label{lem:xpxo}
Let $\vv\in\V$ be a unit spacelike vector.   Then:
\begin{align*}
\vv\boxtimes\vv^+ &=\vv^+ \\
\vv^- \boxtimes\vv &=\vv^- .
\end{align*}
\end{lemma}
\noindent
For the proof, see Charette-Drumm~\cite{CharetteDrumm2}.

Let $\iso$ denote the group of all affine transformations that
preserve the Lorentzian scalar product on the space of directions;
$\iso$ is isomorphic to $\oto\ltimes\V$.  We shall restrict
our attention to those transformations whose linear parts are in
$\soto$, thus preserving orientation and time-orientation. As
above, $\LL$ denotes the projection onto the {\em linear part} of
an affine transformation.

Suppose $g\in\soto$ and $g\neq \Id$.
\begin{itemize}
\item $g$ is {\em hyperbolic} if it has three distinct real
eigenvalues;
\item $g$ is {\em parabolic} if its only eigenvalue is 1;
\item $g$ is {\em elliptic} if it has no real eigenvalues.
\end{itemize}
Denote the set of hyperbolic elements in $\soto$ by
$\Hyp_0$ and the set of parabolic elements by $\Par_0$.

We also call $\gamma\in\iso$ {\em hyperbolic}
(respectively  {\em parabolic}, {\em elliptic})
if its linear part $\LL(\gg)$ is
hyperbolic (respectively  parabolic, elliptic).
Denote the set of hyperbolic elements in $\iso$ by
$\Hyp$ and the set of parabolic transformations by $\Par$.

Let $\gg\in \Hyp \cup \Par$.  The eigenspace $\Fix(\LL(\gg))$ is one-dimensional.
Let $\vv\in\Fix\big(\LL(\gg)\big)$ be a non-zero vector and $x\in\E$.
Define:
\begin{equation*}
\na{\vv}(\gg) \; := \;\ldot{( \gg(x)-x)}{\vv}.
\end{equation*}

\noindent
The following facts are proved in \cite{Abels,CharetteDrumm1,
CharetteDrumm2,DrummGoldman2,G0,GM}:

\begin{itemize}
\item
$\na{\vv}(\gg)$ is independent of $x$;
\item $\na{\vv}(\gamma)$ is identically zero if and only if $\gamma$
fixes a point;\label{hpfact:neq0}
\item For any $\eta\in\iso$ with $h=\LL(\eta)$,
\begin{equation*}
\na{h(\vv)}(\eta\gamma\eta^{-1}) \,=\, \na{\vv}(\gamma)
\end{equation*}
where $\vv\in \Fix(g)$ and $h=\LL(\eta)$;
\item For any $n\in\Z$,
\begin{equation*}
\na{\vv}(\gg^{n}) \,=\, \vert n\vert  \na{\vv}(\gg).
\end{equation*}
\end{itemize}
\noindent
A linear transformation $g$ induces a
natural orientation on $\Fix(g)$ as follows.
\begin{defn}\label{positive}
Let $g\in\Hyp_0\cup\Par_0$. A vector $\vv\in\Fix(g)$ is
{\em positive relative to $g$\/} if and only if
\begin{equation*}
(\vv, \vx, g\vx )
\end{equation*}
is a positively oriented basis, where $\vx$ is any null or
timelike vector which is not an eigenvector of $g$.
\end{defn}
\noindent
The {\em sign of $\gamma$} is the sign of $\na{\vv}(\gamma)$,
where $\vv$ is any positive vector in $\Fix(g)$.
For $n<0$  the orientation of $\Fix(g^n)$
reverses,  so $\gamma$ and $\gg^{-1}$ have equal sign.

\begin{lemma}
[\cite{Margulis1,Margulis2,CharetteDrumm1}]\label{lem:opposite}
Let $\gg_1$, $\gg_2\in \Hyp \cup \Par$ and suppose $\gg_1$ and
$\gg_2$ have opposite signs. Then $\langle \gg_1, \gg_2\rangle$
does not act properly on $\E$.
\end{lemma}

Let $\G_0\subset \oto$ be a free group
and $\rho$ an affine deformation
of $\G_0$:
\begin{equation}\label{eqn:affdef}
\rho(g) (x) = g(x) + u(g)
\end{equation}
where $x\in\V$. Then $\G_0\xrightarrow{u}\V$ is a cocycle of $\G_0$ with
coefficients in the $\G_0$-module $\V$ corresponding to the linear
action of $\G_0$.  As affine
deformations of $\G_0$ correspond to cocycles in $\ZZ$, translational conjugacy classes of affine deformations
comprise the cohomology group $\HH$.

If $g\in\Hyp_0$, set $\vo{g}$ to be the unique positive vector in $\Fix(g)$ such that $\ldot{\vo{g}}{\vo{g}}=1$.  If $g\in\Par_0$, choose a positive vector in $\Fix(g)$ and call it $\vo{g}$.

Let $u\in\ZZ$. Reinterpreting the Margulis invariant as a linear
functional on the space of cocycles $\ZZ$, set:
\begin{align*}
\G_0&\xrightarrow{\alpha_{[u]}}  \R\\
                g&\longmapsto \na{\vo{g}}(\gg),
\end{align*}
\noindent
where $\gg=\rho(g)$ is the affine deformation corresponding to $u(g)$.
As the notation indicates, $\alpha_{[u]}$ only depends on the cohomology class of $u$,
since $\na{\vo{g}}$ is a class function.

\section{Hyperbolic geometry and the three-holed sphere}
\label{sec:THS}

Let $\THS$ denote a complete hyperbolic surface homeomorphic to a
three-holed sphere.  Each of the three ends can either flare out
(that is, have infinite area) or end in a cusp.  In the former
case, a loop going around the end will have hyperbolic holonomy,
and parabolic holonomy in the latter case. We consider certain
geodesic laminations on the surface from which we will construct
crooked fundamental domains.

Fixing some arbitrary basepoint in $\THS$, let $\G_0$ denote the
image under the holonomy representation of the fundamental group
of $\THS$. We may thus identify $\THS$ with $\Ht/ \G_0$.

The fundamental group of $\THS$ is free of
rank two and admits a presentation
\begin{equation}\label{eq:present}
\G_0=\langle g_1,g_2,g_3~\mid~g_3g_2g_1=\Id\rangle,
\end{equation}
where the $g_i$ correspond to the components of $\partial\Sigma$
and  may be hyperbolic or parabolic.

For the rest of the paper, unless otherwise noted, the $g_i$ and their
affine deformations $\gg_i$ are indexed by $i= 1,2,3$ with addition in
$\Z/3\Z$.

If $g_i$ is hyperbolic, it admits a unique invariant axis $l_i\subset\Ht$
which projects to an end of the three-holed sphere.  For $g_i$
parabolic, we think of this invariant line as shrunk to a point on
the ideal boundary.  For hyperbolic $g_i$, set $\vp{i}$, $\vm{i}$ to be its attracting and repelling fixed points,
respectively; if $g_i$ is parabolic, set $\vp{i}=\vm{i}$ to be its unique fixed point.

Since $\G_0$ is discrete, the $l_i$'s are
pairwise disjoint. Furthermore, substituting inverses if
necessary, we assume for convenience that the direction of
translation along the axes is as in Figure~\ref{fig:axes}.  (In
this case, all three $g_i$'s are hyperbolic.)

\begin{figure}
\centerline{\input{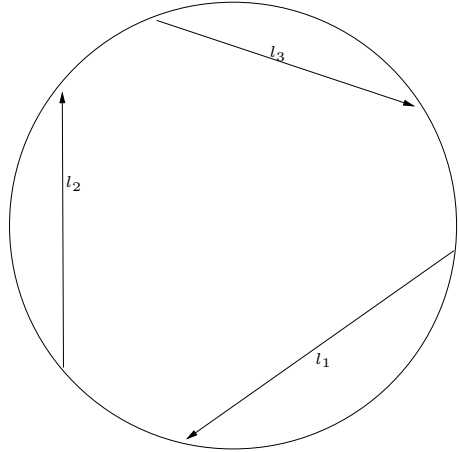}} \caption{The invariant
lines for $g_1,g_2,g_3$, with direction indicated by the arrows.}
\label{fig:axes}
\end{figure}
The three arcs in $\Ht$ respectively joining $\vp{i}$  to
$\vp{i+1}$ project to a geodesic lamination of $\THS$ as drawn in
Figures~\ref{fig:ideal1} and~\ref{fig:ideal2}.

 \begin{figure}
 \centerline{\input{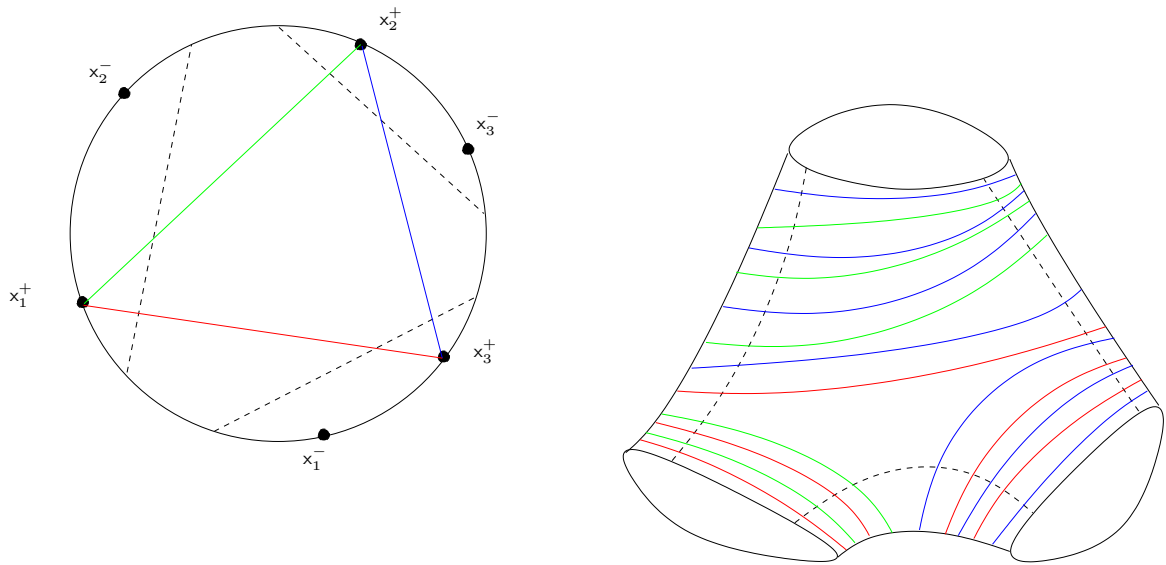}}
 \caption{Three lines in $\Ht$ joining endpoints of the invariant axes $l_i$.
 On the right, the induced lamination of $\THS$.}
 \label{fig:ideal1}
 \end{figure}

 \begin{figure}
 \centerline{\input{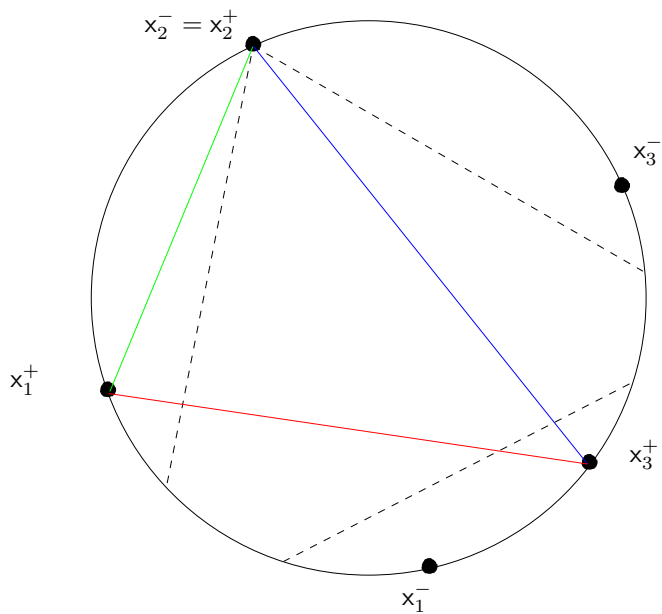}}
 \caption{Three lines in $\Ht$ joining endpoints of $l_i$, with $g_2$
 parabolic and $l_2$ an ideal point.}
 \label{fig:ideal2}
 \end{figure}

We shall adopt the following model for $\Ht$ in terms of Lorentzian affine space
$\E$. A {\em future-pointing timelike ray\/} is a ray
$q + \R_+ \vw$, where $q\in\E$ and $\vw\in\V$ is a future-pointing
timelike vector. Parallelism defines an equivalence relation on
future-pointing timelike rays, and points of $\Ht$ identify with
equivalence classes of future-pointing timelike rays.

Denote by $[q + \R_+ \vw]$ the point in $\Ht$ corresponding to the equivalence class of the ray $q + \R_+ \vw$.

Geodesics in $\Ht$ identify with parallelism classes of indefinite affine planes;
a point in $\Ht$ is incident to a geodesic if and only if the corresponding
future-pointing timelike ray and indefinite affine plane are parallel.
A half-space $H$ in $\E$ bounded by an indefinite affine plane determines
a half-plane $\HS\subset \Ht$. A point $[q + \R_+ \vw]$ in $\Ht$
lies in $\HS$ if and only
if  $q+\R_+\vw$ intersects $H$ in a ray, that is, $q + t\vw \in H$
for $t \muchbigger 0$.

Dually, geodesics in $\Ht$ correspond to spacelike lines, since the
Lorentz-orthogonal plane of a spacelike vector is indefinite.  In
fact, if $l=\R\vo{g_i}$, then the null vectors $\ypm{(\vo{g_i})}$
respectively project to the ideal points $\vpm{i}$.

Furthermore spacelike vectors correspond to oriented geodesics, 
or, equivalently, to half-planes in $\Ht$.
A spacelike vector spans a unique spacelike ray, which contains
a unique unit spacelike vector $\vv$. The corresponding half-plane
is
\begin{equation*}
\HS(\vv) := \{ [q + \R_+ \vw]\in\Ht \mid \ldot{\vw}{\vv} \ge 0  \} .
\end{equation*}
Extending terminology from $\Ht$ to $\V$, say that two spacelike
vectors $\vu,\vv\in\V$ are:
\begin{itemize}
\item {\em ultraparallel} if $\vu\boxtimes\vv$ is spacelike;
\item {\em asymptotic} if $\vu\boxtimes\vv$ is null;
\item {\em crossing} if  if $\vu\boxtimes\vv$ is lightlike.
\end{itemize}

\section{Crooked planes and half-spaces}\label{sec:cp}

Crooked planes are Lorentzian analogs of equidistant surfaces.  We
will think of a triple of crooked planes as the natural extension
of a lamination. We will see how to get pairwise disjoint crooked
plane triples, yielding proper affine deformations of the linear
holonomy. In this section, we define crooked planes and discuss
criteria for disjointness.

Here is a somewhat technical, yet important, point.  What we call
crooked planes  and half-spaces should really be called 
{\em positively extended\/} crooked planes and half-spaces.  
We require
crooked planes to be positively extended when the signs of the
Margulis invariants are positive. But for the case of negative
Margulis invariants, we must use 
{\em negatively extended crooked planes\/}. As the arguments are essentially the same up to a change
in sign change,  for the rest of the paper we will restrict to the
case of positive signs. The curious reader should 
consult~\cite{DrummGoldman1}. 
(In that paper the crooked planes are called
positively or  negatively {\em oriented\/}).

Given a null vector $\vx\in\V$, set $\pp(\vx)$ to be the set of
(spacelike) vectors $\vw$ such that $\yp{\vw}$ is parallel to $\vx$.
This half-plane in the Lorentz-orthogonal plane $\vx^\perp$ is a connected
component of $\vx^\perp \setminus \langle\vx\rangle$.
If $\vv$ is a spacelike vector, then
\begin{align*}
\vv&\in\pp(\yp{\vv}) \\
-\vv& \in\pp(\ym{\vv}).
\end{align*}

Let $p\in\E$ be a point and $\vv\in\R^{2,1}$ a spacelike vector.
Define the {\em crooked plane}
$\CP(\vv,p)\subset\E$ with {\em vertex\/} $p$ and
{\em direction vector\/} $\vv$
to be the union of two {\em wings}
\begin{align*}
& p +\pp(\yp{\vv})\\
& p +\pp(\ym{\vv})
\end{align*}
and a {\em stem}
\begin{equation*}
p +\  \{\vx\in\V \mid\ \ldot{\vv}{\vx} = 0,
\ldot{\vx}{\vx} \le 0 \} .
\end{equation*}
Each wing is a half-plane, and the stem is the union of two
quadrants in a spacelike plane. The crooked plane itself is a
piecewise linear submanifold, which stratifies into four connected
open subsets of planes (two wings and the two components of the
interior of the stem), four null rays, and a vertex.

\begin{defn}
Let $\vv$ be a spacelike vector and $p\in\E$.
The {\em crooked half-space} determined by $\vv$ and $p$,
denoted $\H(\vv,p)$, consists of all $q\in\E$ such that:
\begin{itemize}
\item $\ldot{(q-p)}{\yp{\vv}}\leq0$ if $\ldot{(q-p)}{\vv}\geq 0$;
\item $\ldot{(q-p)}{\ym{\vv}}\geq0$ if $\ldot{(q-p)}{\vv}\leq 0$;
\item  Both conditions must hold for $q-p\in\vv^\perp$.
\end{itemize}
\end{defn}
\noindent
Observe that $\CP(\vv,p)=\CP(-\vv,p)$. In contrast, the crooked half-spaces
$\H(\vv,p)$ and $ \H(-\vv,p)$ are distinct spaces.
Their  union and intersection are respectively:
\begin{align*}
\H(\vv,p)\,\cup\, \H(-\vv,p) & =\E \\
\H(\vv,p)\,\cap\, \H(-\vv,p) & =\CP(\vv,p)= \CP(-\vv,p).
\end{align*}

Crooked half-spaces in $\E$ determine half-planes in $\Ht$ as follows.
As in the preceding section, a point in $\Ht$ corresponds to
the equivalence class of a future-pointing timelike ray.

\begin{lemma}\label{lem:halfspace}
Let $p,q\in\E$ and  $\vv,\vw\in\V$ spacelike.  Suppose that
$\H(\vv,p)$ is a crooked  half-space and that
$\ldot{\vw}{\vv}\neq0$. Then $q + t\vw\in
\int\left(\H(\vv,p)\right)$ for $t\muchbigger 0$ if and only if
$[q + t\vw]\in \int\left(\HS(\vv)\right)$.
\end{lemma}
\begin{proof}
It suffices to consider the case that $p = 0$ and
\begin{equation*}
\vv = \bmatrix 1 \\ 0 \\ 0 \endbmatrix,
\end{equation*}
that is,
\begin{equation*}
\H(\vv,p) \;=\; \left\{ \bmatrix x \\ y \\ z \endbmatrix \;\bigg{|}\;
y + z \ge 0  \text{~if~} x \ge 0  \\
\text{~and~}
y - z \ge 0  \text{~if~} x \le 0
\right\}.
\end{equation*}
By applying an automorphism preserving $\H(\vv,p)$, we may assume
\begin{equation*}
q = \bmatrix x_0 \\ y_0 \\ z_0 \endbmatrix, \quad
\vw = \bmatrix d  \\ 0 \\ 1 \endbmatrix.
\end{equation*}
where $\vert d \vert < 1$.

Set $q(t) := q + t\vw$.  For any value of $d$, $q(t)$ eventually
satisfies 
\begin{equation*}
y + z = y_0 + z_0 + t > 0
\end{equation*}
for $t\muchbigger 0$ and $ y -z<0$.
The point $[q + t\vw]$ lies in the interior $\int\left(\HS(\vv)\right)$ when $d > 0$.
In this case,
$q(t)$ eventually satisfies 
\begin{equation*} 
x = x_0 + t d > 0. 
\end{equation*} 
Thus
$q(t)\in \int\left(\H(\vv,p)\right)$.

Conversely, if $[q + t\vw] \in \HS(-\vv)$, then $d<0$.
If 
$t\muchbigger0$, then $ x< 0$.
Therefore $q(t)\notin \int\left(\H(\vv,p)\right)$ as desired.
\end{proof}

\section{Disjointness of crooked half-spaces}


By~\cite{Drumm1,Drumm2} (see \cite{CG} for another exposition),
the complement of a disjoint union of crooked half-spaces
with pairwise identifications of its boundary defines a fundamental
polyhedron for the group generated by the identifications. 
This section develops criteria for when two crooked half-spaces
are disjoint.
Lemma~\ref{lem:inhalfspace} reduces disjointness of
crooked half-spaces to disjointness of crooked planes. We need
only consider pairs of crooked half-spaces in the case of
ultraparallel or asymptotic vectors: when $\vu$ and $\vv$ are
crossing $\CP(\vu,p)$ and $\CP(\vv,p)$ always
intersect~\cite{DrummGoldman1}. Theorem~\ref{thm:UltraCPCP} and
Theorem~\ref{thm:AsCPCP} provide criteria for disjointness for
crooked planes, and  were established in \cite{DrummGoldman1}.
Their respective corollaries, Corollary~\ref{cor:seammove} and
Corollary~\ref{cor:seammoveII}, provide more useful criteria in
terms of the direction vectors.

 \begin{defn}
\label{def:co} Spacelike vectors $\vv_1, \dots, \vv_n\in\V$
are
{\em consistently oriented} if and only if, whenever $i\neq j$,
\begin{itemize}
\item $\ldot{\vv_i}{\vv_j}<0$;
\item $\ldot{\vv_i}{\ypm{\vv_j}}\leq 0$.
\end{itemize}
\end{defn}
\noindent The second requirement implies that the $\vv_i$ are
pairwise ultraparallel or asymptotic. Equivalently,
$\vv_i,\vv_j,i\neq j$ are consistently oriented if and only if the
interiors of the half-planes $\HS(\vv_i)$ and $\HS(\vv_j)$ are
pairwise disjoint. (See~\cite{G}, \S 4.2.1 for details.)

\begin{lemma}\label{lem:inhalfspace}
Suppose $\vu,\vv$ are consistently oriented, $p\in\E$ and
$\CP(\vu,p)$ and $\CP(\vv,p+\vw)$ are disjoint. Then
$\CP(\vv,p+\vw)\subset  \H(-\vu,p)$.
\end{lemma}
\begin{proof}
Because
\begin{equation*}
\E\setminus \CP(\vv,p)\;=\; \int\left(\H(\vu,p)\right)\,\cup\,  \int\left(\H(-\vu,p)\right),
\end{equation*}
either $\CP(\vv,p+\vw)\subset  \H(\vu,p)$ or
$\CP(\vv,p+\vw)\subset  \H(-\vu,p)$.

Suppose that $\CP(\vv,p+\vw)\subset  \H(\vu,p)$.
The future-pointing timelike rays on
$\CP(\vv,p+\vw)$ lie on the stem of
$\CP(\vv,p+\vw)$ and correspond to the geodesic
$\partial\HS(\vv)$.

Since a future-pointing timelike ray on
$\CP(\vv,p+\vw)$ lies entirely in
$\H(\vu,p)$,
Lemma~\ref{lem:halfspace} implies
that
$$
\partial\HS(\vv) \subset
\HS(\vu).
$$
Since $\vu,\vv$ are consistently oriented,
the half-spaces
$\HS(\vu)$ and $\HS(\vv)$ are disjoint,
and $\HS(\vv)\subset \HS(-\vu)$, a contradiction.
Thus $\CP(\vv,p+\vw)\subset  \H(-\vu,p)$ as desired.
\end{proof}

\begin{thm} \label{thm:UltraCPCP}
Let $\vv_1$ and $\vv_2$ be consistently oriented, ultraparallel,
unit spacelike vectors and $p_1,p_2\in\E$.  The crooked
planes $\CP(\vv_1,p_1)$ and $\CP(\vv_2,p_2)$ are disjoint if and only if
\begin{equation}\label{eq:DisUltraCPCP}
\ldot{(p_2-p_1)}{(\vv_1\boxtimes\vv_2)} >
\vert\ldot{(p_2-p_1)}{\vv_2}\vert  +
\vert\ldot{(p_2-p_1)}{\vv_1}\vert .
\end{equation}
\end{thm}

\begin{cor}\label{cor:seammove}
Let $\vv_1, \vv_2\in\V$ be consistently oriented, ultraparallel
vectors.  Suppose
\begin{equation*}
p_i=a_i\ym{\vv_i}-b_i\yp{\vv_i},
\end{equation*}
for $a_i,b_i>0$, $i=1,2$.  Then $\CP(\vv_1,p_1)$ and $\CP(\vv_2,p_2)$ are
disjoint.
\end{cor}
\begin{proof}
Rescaling if necessary, assume 
$\vv_1,~\vv_2$ are unit spacelike. 
By Lemmas~\ref{lem:crossprod} and~\ref{lem:xpxo},
\begin{align*}
\ldot{\yp{\vv_i}}{(\vv_i\boxtimes\vv_j)}& \;=\;
 \ldot{\yp{\vv_i}}{\vv_j} \\
\ldot{\ym{\vv_i}}{(\vv_i\boxtimes\vv_j)}& \;=\;
-\ldot{\ym{\vv_i}}{\vv_j}.
\end{align*}
for $i\neq j$.
Consequently:
\begin{align}\label{eq:UltraLHS}
\ldot{(p_2-p_1)}{(\vv_1\boxtimes\vv_2)} &\;=\;-\ldot{(a_2\ym{\vv_2}
+b_2\yp{\vv_2})}{\vv_1} \,-\,\ldot{(a_1\ym{\vv_1}+b_1\yp{\vv_1})}{\vv_2}\notag\\
&\;=\; -\ldot{a_2\ym{\vv_2}}{\vv_1} \,-\, \ldot{b_2\yp{\vv_2}}{\vv_1}
\,-\, \ldot{ a_1\ym{\vv_1}}{\vv_2} \,-\,  \ldot{b_1\yp{\vv_1}}{\vv_2}  \notag
\\ &\;>\; \vert\ldot{(a_2\ym{\vv_2} -b_2\yp{\vv_2})}{\vv_1}\vert \notag
\,+\,\vert\ldot{ ( a_1\ym{\vv_1} -b_1 \yp{\vv_1} )}{\vv_2}\vert.
\end{align}
The above inequality follows because each term in the previous
expression is positive (since $\vv_1,\vv_2$ are consistently
oriented). Finally:
\begin{align*}
\vert\ldot{(p_2-p_1)}{\vv_2}\vert  & =\; \vert\ldot{(a_1\ym{\vv_1} -b_1 \yp{\vv_1})}{\vv_2}\vert  \\\
\vert\ldot{(p_2-p_1)}{\vv_1}\vert  & =\; \vert\ldot{(a_2\ym{\vv_2}
-b_2 \yp{\vv_2})}{\vv_1}\vert .
\end{align*}
\end{proof}
\noindent Alternatively, $\CP(\vv_1,p_1)$ and $\CP(\vv_2,p_2)$ are
disjoint if and only if $p_2-p_1$ lies in the cone spanned by the
four vectors
\begin{equation*}
\ym{\vv_2},\; -\yp{\vv_2} ,\; - \ym{\vv_1},\; \yp{\vv_1}.
\end{equation*}
In fact, we allow $a_1=b_1=0$ or
$a_2=b_2=0$ since $p_2-p_1$ would still lie in the open cone.
If three of the four coefficients $a_i,b_i$ are zero, then the crooked
planes intersect in a single point, on the edges of the stems.

Assume now that $\vv_1,\vv_2\in\V$ are consistently oriented,
asymptotic vectors. Assume, without loss of generality:
\begin{equation*}
\ym{\vv_1}=\yp{\vv_2}.
\end{equation*}

\begin{thm}\label{thm:AsCPCP}
Let $\vv_1$ and $\vv_2$ be consistently oriented, asymptotic
vectors  such that $\ym{\vv_1}=\yp{\vv_2}$, and
$p_1,p_2\in\E$. The crooked planes
$\CP(\vv_1,p_1)$ and $\CP(\vv_2,p_2)$ are disjoint if and only if:
\begin{align}
&\ldot{(p_2-p_1)}{\vv_1} \;<\; 0, \notag \\
&\ldot{(p_2-p_1)}{\vv_2} \;<\; 0, \notag\\
&\ldot{(p_2-p_1)}{(\yp{\vv_1} \boxtimes \ym{\vv_2})} \;>\; 0.
\label{eq:asymptotic}
\end{align}
\end{thm}
\noindent As in the ultraparallel case, Theorem~\ref{thm:AsCPCP}
provides criteria for when $\CP(\vv_1,p_1)$ and
$\CP(\vv_2,p_2)$ are disjoint.

\begin{cor}\label{cor:seammoveII}
Let $\vv_1, \vv_2\in\V$ be consistently oriented, asymptotic
vectors such that $\ym{\vv_1}=\yp{\vv_2}$. Suppose
\begin{equation*}
p_i=a_i\ym{\vv_i} -b_i \yp{\vv_i},
\end{equation*}
where $a_i,b_i>0$ for $i=1,2$.
Then $\CP(\vv_1,p_1)$ and $\CP(\vv_2,p_2)$ are disjoint.
\end{cor}
\begin{proof}
Set
\begin{equation*}
\ym{\vv_i}\boxtimes\yp{\vv_i}=\kappa_i^2\vv_i,
\end{equation*}
for  $i=1,2$. Then:
\begin{align*}
\ldot{(p_2-p_1)}{\vv_1} & =a_2\ldot{\ym{\vv_2}}{\vv_1}<0 \\
\ldot{(p_2-p_1)}{\vv_2} & =b_1\ldot{\yp{\vv_1}}{\vv_2}<0
\end{align*}
and:
\begin{align}\label{eq:asy3}
\ldot{(p_2-p_1)}{\left(\yp{\vv_1} \boxtimes\ym{\vv_2}\right)}
&=\;
-b_2\ldot{\yp{\vv_2}}{\left( \yp{\vv_1}\boxtimes\ym{\vv_2}\right)}
\,-\,  a_1\ldot{\ym{\vv_1}}{\left( \yp{\vv_1}\boxtimes\ym{\vv_2}\right)}\notag \\
&=\; -b_2\kappa_2^2 \left( \ldot{\ym{\vv_1}}{\vv_2}\right)
\,-\, a_1\kappa_1^2\left( \ldot{\yp{\vv_2}}{\vv_1}\right) \notag\\
&>\; 0.
\end{align}
\end{proof}
\noindent As in the ultraparallel case, we obtain disjoint crooked
planes if and only if $p_2-p_1$ lies in a cone spanned by three
rays. In Equation~\eqref{eq:asy3}, we allow $b_2=0$ or $a_1=0$
simply because $\ym{\vv_1}=\yp{\vv_2}$.  If $a_2=0$, $b_1=0$ or
$a_1=b_2=0$, then the crooked planes intersect in a null ray.

\section{Crooked fundamental domains}\label{sub:crookedfd}

Now  look at how collections of pairwise disjoint crooked planes
correspond to groups acting properly on $\E$.
Let $\vv$, $\vv'\in\V$ be two spacelike vectors. Suppose
$\gamma\in\iso$ and $p,p'\in\E$ satisfy:
\begin{equation*}
\gamma(\CP(\vv,p))=\CP(\vv',p').
\end{equation*}

Then $\gamma(p)=p'$ and $\LL(\gamma)(\vv)$ is a scalar multiple of
$\vv'$.  In particular, $\gamma\big(\H(\vv,p)\big)$ is one of the
two crooked half-spaces bounded by $\CP(\vv',p')$.

\begin{thm}\label{thm:disjointCPs}
Suppose that
$\H(\vv_i,p_i)$ are $2n$ pairwise disjoint crooked half-spaces and $\gg_1, \ldots \gg_n \in\G$ such that for
all $i$,
\begin{equation*}
\gamma_i\left(\H(\vv_{-i},p_{-i})\right) \;=\;
\E\setminus \mathsf{int} \left(\H(\vv_{i},p_i)\right).
\end{equation*}
Then
$\langle \gg_1, \ldots \gg_n \rangle$ acts freely and properly on
$\E$ with fundamental domain
\begin{equation*}
\E \setminus \bigcup_{-n\leq i\leq n}\mathsf{int}\left(
\H(\vv_i,p_i) \right) .
\end{equation*}
\end{thm}
\begin{proof}
By the assumption
\begin{equation*}
\gamma_i\left(\H(\vv_{-i},p_{-i})\right)
\,=\, \E\,\setminus\mathsf{int} \left(\H(\vv_{i},p_i)\right),
\end{equation*}
the vectors $\vv_{\pm i}$ either cross or are parallel to
$\vo{g_i}$. The theorem is shown in \cite{Drumm1, Drumm2}, assuming, in the case of hyperbolic
$\gg_i$, that the vector $\vv_{i}$ crosses the
fixed vector $\vo{g_i}$. (The vectors $\vv_i$
are parallel to $\vo{g_i}$ for parabolic $\gg_i$.)

However, the methods used in \cite{Drumm1, Drumm2} extend
to the case of hyperbolic generators with $\vv_{\pm i}$ parallel
to $\vo{g_i}$. In particular, the compression of a tubular
neighborhood around lines which touch a boundary crooked plane at
a point in particular transverse directions is bounded from below.
\end{proof}

These fundamental domains notably differ from the standard construction
(as in \cite{Drumm2}).
A crooked fundamental domain $\Delta$ in $\E$ for $\Gamma$ determines
a polygon $\delta$ in $\Ht$ for $\LL(\Gamma)$;
the stems of $\partial\Delta$ define lines in $\Ht$ bounding $\delta$.
However, while $\Gamma\cdot\Delta \,=\, \E$,
the union $\LL(\Gamma)\cdot\delta$ may only be a {\em proper\/} open subset of $\Ht$.
In the present case, this is the universal covering of the interior of the
convex core of $\Sigma$.
The convex core is an incomplete hyperbolic surface bounded
by three closed geodesics.
In contrast, the flat Lorentz manifold $\E/\Gamma$ is {\em complete.}
While the hyperbolic fundamental domains $\LL(\gamma)(\delta)$ only fill
a proper subset of $\Ht$, the crooked fundamental domains $\gamma(\Delta)$ fill
all of $\E$.

Theorem~\ref{thm:disjointCPs} extends to the case when two of the
crooked planes intersect in a single point.

\begin{lemma}
\label{lem:kissing1}
Let $\vv_{-2},\vv_{-1},\vv_1,\vv_2\in\V$ be consistently oriented vectors
and
suppose
$p_{-1},p_1,p_2\in\E$ satisfy:
\begin{align*}
\CP(\vv_{-2},p_{-1})\,\cap\,\CP(\vv_2,p_2) &=\;\emptyset\\
\CP(\vv_{-1},p_{-1})\,\cap\,\CP(\vv_1,p_1) &=\;\emptyset \\
\CP(\vv_{1},p_{1})\,\cap\,\CP(\vv_2,p_2) &=\;\emptyset .
\end{align*}
Then there exists  $p_{-2}\in\E$ such that the
crooked planes
$\CP(\vv_{-2},p_{-2})$ and
$\CP(\vv_2,\gamma_2(p_{-2}))$ are each
disjoint from $\CP(\vv_1,p_1)$.
\end{lemma}
\begin{proof}
Let $\H(\vv_0,p_{-1})$ be the smallest crooked half-space containing
both $\H(\vv_{-2},p_{-1})$ and $\H(\vv_{-1},p_{-1})$.  Then
\begin{equation*}
\H(\vv_0,p_{-1}),\;\H(\vv_1,p_{1}),\; \H(\vv_2,p_{2})
\end{equation*}
are pairwise disjoint. Disjointness of crooked planes is an open
condition. Therefore there exists $\epsilon>0$ such that for any
$\vu\in\V$ of Euclidean norm less than $\epsilon$, the crooked
plane $\CP(\vv_{2},p_{2}+\vu)$ remains disjoint from
$\CP(\vv_{0},p_{-1})$ and $\CP(\vv_{1},p_{1})$.
Corollaries~\ref{cor:seammove} and~\ref{cor:seammoveII} imply the
existence of a $p_{-2}$ such that $\CP(\vv_{-2},p_{-2})$ is
disjoint from both $\CP(\vv_{0},p_{-1})$ and
$\CP(\vv_{-1},p_{-1})$. The set of choices being closed under
positive rescaling, one can choose $p_{-2}$ close enough to
$p_{-1}$ so that $\gamma_2(p_{-2})$ is within an
$\epsilon$-neighborhood of $p_2$.

Lemma~\ref{lem:inhalfspace} implies:
\begin{equation*}
\CP(\vv_{-2},p_{-2})\subset\H(\vv_0,p_{-1}).
\end{equation*}
In particular,
$\CP(\vv_{-2},p_{-2})$ is disjoint from
each
$\CP(\vv_1,p_1)$ and
$\CP(\vv_2,\gamma_2(p_{-2}))$
as claimed.
\end{proof}

\section
{The space of proper affine deformations} \label{sec:cocycles}

Recall the presentation of the fundamental group of $\THS$ in
Equation~\eqref{eq:present}.    We parametrize the space of
translational conjugacy classes $\HH$ of affine deformations of
$\G_0$ by Margulis invariants corresponding to $g_1$, $g_2$,
$g_3$.  Positivity of the three signs will guarantee a triple of
crooked planes arising from the lamination described
in~\S\ref{sec:THS}. (Alternatively, if the signs are all negative,
use negatively extended crooked planes~\cite{DrummGoldman1} as
mentioned in~\S\ref{sec:cp}.) The existence of such a crooked
polyhedron thereby completes the proof of Theorem~A.

We begin with the parametrization of  $\HH$.
\begin{lemma}\label{lem:muiso}
Let $\pi$ denote a free group of rank two with presentation
\begin{equation*}
\langle A_1, A_2, A_3 \mid A_1 A_2 A_3  \,=\, \Id \rangle.
\end{equation*}
Let $\pi\xrightarrow{\rho_0}\soto$ be a homomorphism such that
$\rho_0(A_i) \,\in\,  \Hyp_0 \cup \Par_0$ for $i=1,2,3$. 
Suppose that
$\rho(\pi)$ is not solvable. For each $i$ choose a vector
$\vv_i\in \Fix\big(\rho_0(A_i)\big)$ positive with respect to
$\rho_0(A_i)$ and define
\begin{align*}
\HH &\xrightarrow{\mu_i} \R \\
[u] &\longmapsto\; \na{\vv_i}(\rho(A_i)) \,=\,
u(A_i) \cdot \vv_i
\end{align*}
where $\rho$ is the affine deformation corresponding to $u$.
Then
\begin{align*}
\HH & \xrightarrow{\mu} \R^3 \\
\mu: [u] & \longmapsto
\bmatrix \mu_1([u]) \\ \mu_2([u]) \\ \mu_3([u]) \endbmatrix
\end{align*}
is a linear isomorphism of vector spaces.
\end{lemma}

Of course this lemma is much more general than our specific
application. In our application $\rho_0$ is an isomorphism of $\pi
= \pi_1(\Sigma)$ onto the discrete subgroup $\G_0\subset\soto$,
and corresponds to a complete hyperbolic three-holed sphere
$\mathsf{int}(\Sigma)$. The generators $A_1,A_2,A_3$ correspond to
the three components of $\partial\Sigma.$

The proof of Lemma~\ref{lem:muiso} is postponed to the Appendix.

As in \S\ref{sec:alpha}, choose a positive vector
$\vo{i}:=\vo{g_i}\in\Fix(g_i)$,  further requiring that $\vo{i}$
be unit spacelike when $g_i$ is hyperbolic.  With this fixed
choice of positive vectors:
\begin{equation*}
\mu_i([u])
\;=\alpha_{[u]}(g_i).
\end{equation*}

We will now show that every positive cocycle
$(\mu_1,\mu_2,\mu_3)\in\ZZ$ corresponds to a  triple of mutually
disjoint crooked planes arising from the geodesic lamination described
 in~\S\ref{sec:THS}.

By a slight abuse of notation, set $\vpm{i}=\ypm{(\vo{i})}$
and $\vp{i}=\vm{i}=\vo{i}$ when $g_i$ is parabolic.
The three consistently oriented unit spacelike vectors
\begin{equation*}
\vv_i\;=\;\frac{-1}{\vp{i}\cdot\vp{i+1}}\,\vp{i}\boxtimes\vp{i+1}
\end{equation*}
correspond to the arcs joining $\vp{i}$ to $\vp{i+1}$ in $\Ht$.

\begin{lemma}
For $i=1,2,3$, choose $a_i,b_i>0$.
For
\begin{equation*}
p_i\;:=\; a_i\vp{i}-b_i\vp{i+1}
\end{equation*}
the crooked planes $\CP(\vv_i,p_i)$ are pairwise disjoint.
\end{lemma}

\begin{proof}
Each pair being asymptotic, we verify condition
\eqref{eq:asymptotic} in Theorem~\ref{thm:AsCPCP}. We check this
for $\CP(\vv_1,p_1)$ and $\CP(\vv_2,p_2)$; the other cases follow
from cyclic symmetry.

\begin{itemize}
\item $(p_2-p_1)\cdot\vv_1>0$:
\begin{align*}
(p_2-p_1)\cdot\vv_1& = \frac{-1}{\vp{1}\cdot\vp{2}}
(a_2\vp{2}-b_2\vp{3}-a_1\vp{1}+b_1\vp{2})\cdot
(\vp{1}\boxtimes\vp{2}) \\
& =\frac{b_2}{\vp{1}\cdot\vp{2}}\vp{3}\cdot
(\vp{1}\boxtimes\vp{2})>0.
\end{align*}
\item $(p_2-p_1)\cdot\vv_2>0$:
\begin{equation*}
(p_2-p_1)\cdot\vv_2 = \frac{a_1}{\vp{1}\cdot\vp{2}} \vp{1}\cdot
(\vp{2}\boxtimes\vp{3}) >0.
\end{equation*}
\item $(p_2-p_1)\cdot (\vp{1}\boxtimes\vp{2})>0$:
\begin{equation*}
(p_2-p_1)\cdot (\vp{1}\boxtimes\vp{2}) =-b_2\vp{3}\cdot
(\vp{1}\boxtimes\vp{2}) >0.
\end{equation*}
\end{itemize}
\end{proof}

\begin{proof}[Conclusion of proof of Theorem~A]
Consider the following four crooked planes:
\begin{align*}
&\CP(\vv_3,p_3),~\CP(g_1(\vv_3),p_1)\subset\H(\vv_1,p_1)\\
&\CP(\vv_2,p_2),~\CP(g_2^{-1}(\vv_2),p_1)\subset\H(\vv_1,p_1)
\end{align*}
Then apply Lemma~\ref{lem:kissing1} to obtain a
crooked fundamental domain  for the cocycle $u$ such that
$u(g_i)=p_i-p_{i-1}$, $i=1,2,3$.

Every positive cocycle arises in this way.
Indeed, compute the Margulis invariant for the above cocycle $u$:
\begin{eqnarray*}
\mu_1  &= &(p_1-p_3)\cdot\vo{1}\\
&=&(a_1\vp{1}-b_1\vp{2}-a_3\vp{3}+b_3\vp{1})
\cdot\frac{-(\vm{1}\boxtimes\vp{1})}{\vm{1}\cdot\vp{1}}\\
& = &(-b_1\vp{2}-a_3\vp{3})\cdot\vo{1}
\end{eqnarray*}
(Omit the second line when $g_1$ is parabolic.)

Recall that every product  $\beta_{i,j}= -\vp{i}\cdot\vo{j}>0$.
In matrix form:
\begin{equation*}
\begin{bmatrix} \mu_1 \\ \mu_2 \\ \mu_3 \end{bmatrix} =
\begin{bmatrix} 0 & \beta_{2,1} & 0 & 0  & \beta_{3,1}&  0 \\
\beta_{1,2}& 0 & 0 &\beta_{3,2}  & 0 & 0  \\
0 & 0 &\beta_{2,3} & 0 & 0 &  \beta_{1,3}\end{bmatrix}
\begin{bmatrix} a_1 \\ b_1 \\ a_2 \\ b_2 \\ a_3 \\ b_3
\end{bmatrix}.
\end{equation*}
and every positive triple of values $(\mu_1,\mu_2,\mu_3)$ may be
realized by choosing appropriate positive values of $a_i,b_i$.
Explicitly, for $i=1,2,3$, choose $p_i,q_i>0$ with $p_i + q_i=1$,
and define
\begin{equation*}
\begin{bmatrix} a_1 \\ b_1 \\ a_2 \\ b_2 \\ a_3 \\ b_3\end{bmatrix} \;=\;
\begin{bmatrix}
p_2\mu_2/\beta_{12} \\ q_1\mu_1/\beta_{21} \\
p_3\mu_3/\beta_{23} \\ q_2\mu_2/\beta_{32} \\
p_1\mu_1/\beta_{31} \\ q_3\mu_3/\beta_{13}
\end{bmatrix}.
\end{equation*}
The proof of Theorem~A is complete.\end{proof}

\section
{Embedding in an arithmetic group}
As an application, we construct examples of proper affine deformations
of a Fuchsian group as subgroups of the symplectic group $\Spfr$.

Consider a $4$-dimensional real symplectic vector space $S$ with a
lattice $S_\Z$ such that the symplectic form takes values $\Z$ on
$S_\Z$. Fix an {\em integral\/} Lagrangian $2$-plane $\Li\subset
S$, that is, a Lagrangian $2$-plane generated by $\Li\cap S_\Z$.
Our model for Minkowski space will be the space $\Lf$ of all
Lagrangian $2$-planes $L\subset S$  transverse to $\Li$. The
underlying Lorentzian vector space is the space of linear maps
$S/\Li\rightarrow \Li$ which are {\em self-adjoint\/} in the sense
described below. We denote by $\Aut(S)\cong \Spfr$ the group of
linear {\em symplectomorphisms\/} of $S$. We denote by
$\Aut(\Li)\cong\GLtwR$ the group of linear automorphisms of the
vector space $\Li$.

The set of two-dimensional subspaces $L\subset S$ transverse to
$\Li$ admits a simply transitive action of the vector space
$\Hom(S/\Li,\Li)$, as follows. Denote the inclusion and quotient
mappings by
\begin{equation*}
\Li \stackrel{\iota}\hookrightarrow S  \stackrel{\Pi}\twoheadrightarrow S/\Li
\end{equation*}
respectively. Let $L$ be a $2$-plane transverse to $\Li$ and
$\phi\in\Hom(S/\Li,\Li)$. Define the action $\phi\cdot L$ of
$\phi$ on $L$ as the
 {\em graph\/} of the composition
\begin{equation*}
L \xrightarrow{\Pi} S/\Li \xrightarrow{\phi} \Li
\stackrel{\iota}\hookrightarrow S,
\end{equation*}
that is,
\begin{equation*}
\phi\cdot L \;:=\; =  \{ v + \iota\circ\phi\circ\Pi(v) \mid v\in L \}.
\end{equation*}
The vector group $\Hom(S/\Li,\Li)\cong\R^4$ acts simply
transitively on the set of $2$-planes $L$ transverse to $\Li$ as
claimed.

Such a $2$-plane $L$ is Lagrangian if and only if the
corresponding linear map $\phi$ is {\em self-adjoint\/} as
follows. Since $S$ is 4-dimensional and $\Li\subset S$ is
Lagrangian, the symplectic structure on $S$ defines an isomorphism
of $S/\Li$ with the dual vector space $\Listar$. Let
$\phi\in\Hom(S/\Li,\Li)$ be a linear map. Its {\em transpose\/}
$\phi^\mathsf{T} \in\Hom\big(\Li^*,(S/\Li)^*\big)$  is the map
induced by $\phi$ on the dual spaces. Its {\em adjoint\/}
$\phi^*\in\Hom(S/\Li,\Li)$ is defined as the composition
\begin{equation}
S/\Li \xrightarrow{\cong} \Listar
\xrightarrow{\phi^\mathsf{T}}  (S/\Li)^*
\xrightarrow{\cong}  \Li
\end{equation}
and the isomorphisms above arise from duality between $S/\Li$ and
$\Li$. If $L\in \Lf$, and $\phi\in\Hom(S/\Li,\Li)$, then
$\phi\cdot L$ is Lagrangian if and only if $\phi = \phi^*$, that
is, $\phi$ is {\em self-adjoint.} In this case $\phi$ corresponds
to a symmetric bilinear form on $S/\Li$.

Let $\Phi\cong\R^3$ denote the vector space of such self-adjoint
elements $\phi$ of $\Hom(S/\Li,\Li)$. Then $\Lf$ is an affine
space with underlying vector space of translations $\Phi$.

Choose a fixed $\Lo\in\Lf$. The symplectic form defines a
nondegenerate bilinear form
\begin{equation*}
\Lo \times \Li \longrightarrow \R
\end{equation*}
under which $\Lo$ and $\Li$ are dual vector spaces and $S = \Li \oplus \Lo$.
The restriction $\Pi|_\Lo$ induces an isomorphism
\begin{equation*}
\Lo \xrightarrow{\cong} S/\Li.
\end{equation*}
Given $\phi\in\Phi$, a self-adjoint endomorphism of $\Hom(S/\Li,\Li)$,
the linear transformation of $S = \Lo \oplus \Li$ defined
by the exponential map
\begin{equation*}
U_\phi \;:=\;  \exp \big(0 \oplus \, (\phi\circ \Pi|_\Lo)\,  \big)
\end{equation*}
is a unipotent linear symplectomorphism of $S$ which:
\begin{itemize}
\item acts identically on $\Li$;
\item induces the identity on the quotient $S/\Li$.
\end{itemize}
Indeed, the exponential map is an isomorphism of the vector group
$\Phi$ onto the subgroup of the linear symplectomorphism group of
$S$ satisfying the above two properties.

 Every linear automorphism $A$ of $\Li$ extends to the
linear symplectomorphism of $S= \Li \oplus \Lo$:
\begin{equation*}
\sigma(A) \;:=\; A \oplus (A^\mathsf{T})^{-1} .
\end{equation*}
Such linear symplectomorphisms stabilize the Lagrangian subspaces
$\Li$ and $\Lo$, and the image of $\Aut(\Li)$ is characterized
by these properties.
In particular  $\Aut(\Li)$ normalizes the group $\exp(\Phi)$
corresponding to translations.
These two subgroups generate the subgroup of linear symplectomorphisms
of $S$ which stabilize $\Li$.

The vector space $\Phi$ has a natural {\em Lorentzian\/} structure
as follows. Identify $\Phi$  with the vector space $\St$ of
$2\times 2$ symmetric matrices. The bilinear form
\begin{align*}
\St \times \St & \longrightarrow \R \\
X\cdot Y &\;\longmapsto\; \frac{\tr\left(XY\right) - \tr(X)\tr(Y)}{2}
\end{align*}
is a Lorentzian inner product of signature $(2,1)$.
If $A\in\Aut(\Li)$, then
\begin{equation*}
A X \cdot A Y = (\det A)^2 X \cdot Y
\end{equation*}
so the subgroup $\SAut(\Li)$ of unimodular automorphisms acts
isometrically with  respect to this inner product. In this way
$\Lf$ is a model for Minkowski space and $\SAut(\Li)$ acts by
linear isometries.  In particular, $\exp(\Phi)$ corresponds to the
group of translations.

We describe this explicitly by matrices.
Consider $\R^4$ with standard basis vectors $\e_k$
for $1\leq k \leq 4$.
Endow $\R^4$ with the symplectic form such that:
\begin{align*}
 \omega\left( \e_1, \e_3\right) \;=\; - \omega\left( \e_3, \e_1\right) & \;=\; 1 \\
 \omega\left(  \e_2, \e_4\right) \;=\; -  \omega\left( \e_4, \e_2\right) & \;=\; 1
 \end{align*}
 and all other $\omega\left( \e_i, \e_j\right) = 0$.
That is,
\begin{equation*}
\omega(u, v) \;:=\; u^{\mathsf{T}} \mathbb{J} v
\end{equation*}
where
\begin{equation*}
\mathbb{J} \;:=\;
\bmatrix 0 & 0 & 1 & 0 \\
0 & 0 & 0 & 1 \\
-1 & 0 & 0 & 0 \\
0 & -1 & 0 & 0 \endbmatrix.
\end{equation*}
 Define the complementary pair of Lagrangian planes:
  \begin{align*}
 \Li & \;:=\; \langle \e_1 , \e_2 \rangle \\
 \Lo & \;:=\; \langle \e_3 , \e_4 \rangle.
\end{align*}
Thus $\left(\e_3,\e_4\right)$ is the basis of $\Lo$ dual to the basis
 $\left(\e_1,\e_2\right)$.

Vectors in Minkowski space correspond to self-adjoint linear
transformations $\Li \rightarrow \Lo \cong \Listar$, that is,
$2\times 2$ symmetric matrices as follows.
A symmetric matrix
\begin{equation*}
\psi(x,y,z) \;:=\; \bmatrix x & y \\ y & z \endbmatrix
\end{equation*}
corresponds to a vector in Minkowski space with quadratic form
\begin{equation*}
-\det(\psi) = x z  - y^2.
\end{equation*}
The unipotent symplectomorphism corresponding to a symmetric
matrix  $\psi(x,y,z)\in\St$  is:
\begin{equation*}
U_{\psi(x,y,z)} \;:=\;
\exp\left(
\bmatrix 0 & 0 & x & y \\
0 & 0 & y & z \\
0 & 0 & 0 & 0 \\
0 & 0 & 0 & 0 \endbmatrix \right)
\;=\;
\bmatrix 1 & 0 & x & y \\
0 & 1 & y & z \\
0 & 0 & 1 & 0 \\
0 & 0 & 0 & 1 \endbmatrix
\end{equation*}
where $x,y,z\;\in\;\R$.
These correspond to the translations of Minkowski space,
and comprise the subgroup $\U\subset\Spfr$.

The reductive subgroup $\Aut(\Li) \;\cong\; \GLtwR$ embeds in
$\Aut(S) \;\cong\; \Spfr$ as follows:
let
\begin{equation*}
A \;:=\; \bmatrix a & b \\ c & d \endbmatrix \;\in\;\  \GLtwR \;\cong\;
\Aut(\Li)
\end{equation*}
with determinant $\Delta \;:=\; \det(A).$
The corresponding linear symplectomorphism
preserving the decomposition $S \;=\; \Li \oplus \Lo$ is:
\begin{equation*}
\sigma(A) \;:=\;
\bmatrix  a & b & 0 & 0 \\
c & d & 0 & 0 \\
0 & 0 &   d/\Delta & -  c/\Delta   \\
0 & 0 & -b/\Delta &  a/\Delta \endbmatrix.
\end{equation*}
These correspond to linear conformal transformations of Minkowski
space. The subgroup $\SAut(\Li)$ of {\em unimodular\/}
automorphisms of $\Li$ corresponds to the group of linear isometries
of Minkowski space.

The subgroup of $\Aut(S)$ generated by $\U$ and $\SAut(\Li)$ is a
semidirect product $\U\rtimes\SAut(\Li)$ and acts by conjugation
on the normal subgroup $\U$. This action corresponds to the action
of the group of affine isometries of Minkowski space.

We construct subgroups of $\Spfz$ which act properly on the $\St$
model of $\E$. The linear parts and translational parts of
Lorentzian transformations of $\St$ are associated with elements
of $\Spfz$. The level two congruence subgroup $\Gamma_0$ of
$\SLtz$ is generated by
\begin{equation*}
g_1 \;:=\;-\begin{bmatrix} 1 & 2 \\ 0 & 1 \end{bmatrix} ,\;
g_2 \;:=\;-\begin{bmatrix} 1 & 0 \\ -2 & 1 \end{bmatrix} ,\;
g_3 \;:=\; \begin{bmatrix} -1 & 2 \\ -2 & 3 \end{bmatrix}.
\end{equation*}
subject to the relation $g_1g_2g_3 = \Id$. It is freely generated
by $g_1$ and $g_2$. All three $g_i$ are parabolic and the quotient
hyperbolic surface $\Sigma \;:=\; \Ht/\Gamma_0$ is a three-punctured sphere.
The symmetric matrices
\begin{equation*}
\vv_1 \;:=\; \bmatrix -2 & 0 \\ 0 & 0 \endbmatrix,\; \vv_2 \;:=\;
\bmatrix 0 & 0 \\ 0 & -2 \endbmatrix,\; \vv_3 \;:=\; \bmatrix -2 &
-2 \\ -2 & -2 \endbmatrix
\end{equation*}
define positive fixed vectors with respect to $g_1,g_2,g_3$ respectively.
The triple $(\vv_1,\vv_2,\vv_3)$ defines a decoration of $\Sigma$.

An affine deformation of $\Gamma_0$ is defined by two arbitrary vectors
$u_1,u_2\in\St$ as translational parts:
\begin{equation*}
u_1 \;:=\;\begin{bmatrix} a_1 & b_1 \\ b_1 & c_1 \end{bmatrix} ,\;
u_2 \;:=\;\begin{bmatrix} a_2 & b_2 \\ b_2 & c_2 \end{bmatrix}.
\end{equation*}
Thus the affine transformations with linear part $g_i$ and translational part $u_i$ are:
\begin{equation*}
\gamma_1 \;:=\;
\begin{bmatrix} 1 & 0 & a_1 & b_1 \\ 0 & 1 & b_1 & c_1 \\
0 & 0 & 1 & 0 \\ 0 & 0 & 0 & 1 \end{bmatrix}
\begin{bmatrix} -1 & -2 & 0 & 0 \\ 0 & -1 & 0 & 0 \\
0 & 0 & -1 & 0 \\ 0 & 0 & 2 & -1 \end{bmatrix}
\end{equation*}

\begin{equation*}
\gamma_2 \;:=\;
\begin{bmatrix} 1 & 0 & a_2 & b_2 \\ 0 & 1 & b_2 & c_2 \\
0 & 0 & 1 & 0 \\ 0 & 0 & 0 & 1 \end{bmatrix}
\begin{bmatrix} -1 & 0 & 0 & 0 \\ 2 & -1 & 0 & 0 \\
0 & 0 & -1 & -2 \\ 0 & 0 & 0 & -1 \end{bmatrix}
\end{equation*}
and
\begin{equation*}
\gamma_3 \;:=\;
\begin{bmatrix} 1 & 0 &  a_3  &b_3 \\
0 & 1 & b_3  &
c_3 \\
0 & 0 & 1 & 0 \\ 0 & 0 & 0 & 1 \end{bmatrix}
\begin{bmatrix} 1 & -2 & 0 & 0 \\ 2 & -3 & 0 & 0 \\
0 & 0 & -3 & -2 \\ 0 & 0 & 2 & 1 \end{bmatrix},
\end{equation*}
where $\gamma_3 \; =\; (\gamma_1 \gamma_2)^{-1}$,
 \begin{itemize}
\item $a_3 =-a_1-a_2+4 b_1 - 4 c_1$,
 \item $b_3 = -2 a_1 -2 a_2 + 7 b_1 - b_2 -6 c_1$, and
 \item $c_3 =-4 a_1 -4 a_2 + 12 b_1 -4 b_2 - 9 c_1 - c_2$.
 \end{itemize}
The corresponding Margulis invariants taken with respect to
$\vv_1,\vv_2,\vv_3$ are:
\begin{align*}
\mu_1 & =\; c_1 \\
\mu_2 & =\; a_2 \\
\mu_3 & =\; c_1 + c_2 - 2 b_1 + 2 b_2 + a_1 + a_2.
\end{align*}
By Theorem~A, the affine deformation $\Gamma :=
\langle \gamma_1,\gamma_2\rangle$ acts properly with
crooked fundamental domain whenever
\begin{align*}
\mu_1 & >\; 0 \\
\mu_2 & >\; 0 \\
\mu_3 & >\; 0.
\end{align*}
Furthermore, taking $a_1,b_1,c_1,a_2,b_2,c_2\;\in\; \Z$
implies $\Gamma\subset\Spfz$.

Here are some explicit examples. Consider, for example the slice
for translational conjugacy defined by $b_1 = b_2 = c_2 = 0$.
Choose three positive integers $\mu_1,\mu_2,\mu_3$.
Take
\begin{align*}
a_1 & =\; \mu_3 - \mu_1 -\mu_2\\
c_1 & =\; \mu_1 \\
a_2 & =\; \mu_2,
\end{align*}
that is, let
\begin{equation*}
\gamma_1 \;:=\;
\begin{bmatrix}
1 & 0 & \mu_3 - \mu_1 -\mu_2 & 0 \\ 0 & 1 & 0 & \mu_1 \\
0 & 0 & 1 & 0 \\ 0 & 0 & 0 & 1 \end{bmatrix}
\begin{bmatrix} -1 & -2 & 0 & 0 \\ 0 & -1 & 0 & 0 \\
0 & 0 & -1 & 0 \\ 0 & 0 & 2 & -1 \end{bmatrix}
\end{equation*}
and
\begin{equation*}
\gamma_2 \;:=\;
\begin{bmatrix} 1 & 0 & \mu_2 & 0 \\ 0 & 1 & 0 & 0 \\
0 & 0 & 1 & 0 \\ 0 & 0 & 0 & 1 \end{bmatrix}
\begin{bmatrix} -1 & 0 & 0 & 0 \\ 2 & -1 & 0 & 0 \\
0 & 0 & -1 & -2 \\ 0 & 0 & 0 & -1 \end{bmatrix}.
\end{equation*}
The proof of Theorem~B is complete.\qed

\appendix
\section*{Appendix. Proof of Lemma~\ref{lem:muiso}}

We return to the parametrization of the cohomology $\HH$ by the three
generalized Margulis invariants $\mu_1,\mu_2,\mu_3$ associated to the
respective generators $g_1,g_2,g_3$ associated to components of
$\partial\Sigma$. When $g_i$ is parabolic, choose a positive
vector $\vv_i$ to define $\mu_i$. We must show that
the triple $\mu = (\mu_1,\mu_2,\mu_3)$
defines an isomorphism
\begin{equation*}
\HH\longrightarrow\R^3.
\end{equation*}

%



Under the double covering $\SLtr \longmapsto \soto$, lift $\rho_0$ to a representation
$\pi\xrightarrow{\trho_0} \SLtr$. The condition that $\rho_0(\pi)$ is not solvable
implies that the representation $\trho_0$ on $\R^2$ is irreducible.
By a well-known classic theorem (see, for example, Goldman~\cite{G}),
such a representation is determined up to conjugacy by the three traces
$$
a_i \; := \; \tr\big( \trho_0(A_i)\big).
$$
and, choosing $b_3$ such that $b_3 + 1/b_3 = a_3$, we may conjugate $\trho_0$ to the
representation defined by:
\begin{align}\label{eq:Aslice}
\trho_0(A_1) &=\; \bmatrix a_1 & -1 \\ 1 & 0 \endbmatrix \notag \\
\trho_0(A_2) &=\; \bmatrix 0  & -b_3 \\ 1/b_3 & a_2 \endbmatrix \notag \\
\trho_0(A_2) &=\; \bmatrix b_3 & -a_1 c_3 + a_2  \\ 0 &  1/b_3  \endbmatrix.
\end{align}
\noindent
Since $\pi$ is freely generated by $A_1,A_2$,
a cocycle $\pi\xrightarrow{u}\V$ is completely
determined by two values $u(A_1),u(A_2)\in\V$.
Furthermore, since $\rho_0(\pi)$ is nonsolvable,
the coboundary map
$$
\V \xrightarrow{\partial} \ZZ
$$
is injective. Therefore the vector space
$\HH$ has dimension three.

To show that the linear map $\mu$ is an isomorphism,
it suffices to show that $\mu$ is onto. To this end,
it suffices to show that for each $i=1,2,3$ there
is a cocycle $u\in\ZZ$ such that $u(A_i) \neq 0$
and $u(A_j) = 0$ for $j\neq i$. By cyclic symmetry
it is only necessary to do this for $i=1$.

Under the local isomorphism $\SLtr \longmapsto \soto$, the Lie algebra
$\mathfrak{sl}(2,\R)$ maps to the Lie algebra $\mathfrak{so}(2,1)$ which
in turn maps isomorphically to the Lorentzian vector space $\V$.
(Compare \cite{GM,G0,CharetteDrummGoldman}.)
If $g\in\SLtr$ is hyperbolic or parabolic, then a neutral eigenvector $\vo(g)$ is
a nonzero multiple of the traceless projection
$$
\hat{g} := g - \frac{\tr(g)}2 \Id.
$$
Define a cocycle for the representation $\trho_0$ defined in \eqref{eq:Aslice} by:
\begin{align*}
u(A_1) &:=\; \bmatrix 1 & 0 \\ 0 & 0 \endbmatrix \\
u(A_2) &:=\; \bmatrix 0 & 0 \\ 0 & 0 \endbmatrix \\
u(A_3) &:=\; \bmatrix 0 & 0 \\ -1/c & 0 \endbmatrix.
\end{align*}
Then $\mu_1(u) \neq 0$ but $\mu_2(u)= \mu_3(u) =0$ as claimed.
The proof of Lemma~\ref{lem:muiso} is complete.\qed

\end{document}